\documentclass[reqno]{amsart}

\usepackage{amsmath,amssymb,amsthm,amsfonts,enumerate, enumitem,hyperref,cleveref}
\usepackage{dsfont}
\usepackage[all]{xy}
\usepackage[T1]{fontenc}
\usepackage{tikz}
\usepackage{pgfplots}
\pgfplotsset{compat=1.15}
\usepackage{mathrsfs}
\usetikzlibrary{arrows}

\DeclareMathOperator{\Span}{span}

\theoremstyle{plain}

\newtheorem{thrm}{Theorem}[section]
\newtheorem{cor}[thrm]{Corollary}
\newtheorem{prop}[thrm]{Proposition}
\newtheorem{lem}[thrm]{Lemma}

\theoremstyle{definition}
\newtheorem{defn}[thrm]{Definition}
\newtheorem{rem}[thrm]{Remark}
\newtheorem{exm}[thrm]{Example}

\crefrangeformat{equation}{#3(#1)#4--#5(#2)#6}

\crefname{thrm}{Theorem}{Theorems}
\crefname{theorem}{Theorem}{Theorems}
\crefname{lem}{Lemma}{Lemmas}
\crefname{cor}{Corollary}{Corollaries}
\crefname{prop}{Proposition}{Propositions}
\crefname{defn}{Definition}{Definitions}
\crefname{exm}{Example}{Examples}
\crefname{rem}{Remark}{Remarks}
\crefname{conj}{Conjecture}{Conjectures}
\crefname{quest}{Question}{Questions}
\crefname{section}{Section}{Sections}
\crefname{equation}{\unskip}{\unskip}
\crefname{enumi}{\unskip}{\unskip}
\crefname{subsection}{Subsection}{Subsections}

\newcommand{\af}{\alpha}
\newcommand{\bt}{\beta}
\newcommand{\lb}{\lambda}

\newcommand{\vf}{\varphi}

\newcommand{\dl}{\delta}

\newcommand{\kp}{\kappa}

\newcommand{\CC}{\mathbb{C}}
\newcommand{\Z}{\mathbb{Z}}
\newcommand{\NN}{\mathbb{N}}

\newcommand{\cP}{\mathcal{P}}

\newcommand{\m}{{}^{-1}}
\newcommand{\sst}{\subseteq}

\newcommand{\impl}{\Rightarrow}

\newcommand{\tr}{\triangle}
\newcommand{\ch}{\mathrm{char}}

\renewcommand{\iff}{\Leftrightarrow}

\begin{document}
	\title[Invertibility preservers of finitary incidence algebras]{Invertibility preservers\\ of finitary incidence algebras}	
	
	\author{Jorge J. Garc{\' e}s}
	\address{Departamento de Matem{\' a}tica Aplicada a la Ingenier{\' i}a Industrial, ETSIDI, Universidad Polit{\' e}cnica de Madrid, Madrid, Spain}
	\email{j.garces@upm.es}
	
	\author{Mykola Khrypchenko}
	\address{Departamento de Matem\'atica, Universidade Federal de Santa Catarina,  Campus Reitor Jo\~ao David Ferreira Lima, Florian\'opolis, SC, CEP: 88040--900, Brazil \and CMUP, Departamento de Matemática, Faculdade de Ciências, Universidade do Porto, Rua do Campo Alegre s/n, 4169--007 Porto, Portugal}
	\email[Corresponding author]{nskhripchenko@gmail.com}

	
	\begin{abstract}
		Let $X$ be an arbitrary poset and $K$ an arbitrary field. We describe linear unital invertibility preservers of the finitary incidence algebra $FI(X,K)$ in terms of certain maps of the power set algebra $\cP(X)$ and linear maps $FI(X,K)\to J(FI(X,K))$.
	\end{abstract}
	
	\maketitle
	
	\tableofcontents
	
	\section*{Introduction}
	
	Given a unital associative algebra $A$, denote by $U(A)$ its group of units. A map $\vf:A\to B$ between unital associative algebras is called an \textit{invertibility preserver} if 
	\begin{align}\label{vf(U(A))-sst-U(B)}
		\vf(U(A))\sst U(B).    
	\end{align}
	An invertibility preserver $\vf$ is \textit{strong} if $\vf(U(A))=U(B)$. If $\vf(1_A)=1_B$, then $\vf$ is said to be \textit{unital}. In the case $A=B$ we say that $\vf$ satisfying \cref{vf(U(A))-sst-U(B)} is an \textit{invertibility preserver of $A$}.
	
	It follows from the classical result by Dieudonn\'e~\cite[Theorem 3]{Dieudonne49} that bijective strong linear invertibility preservers $\vf:M_n(K)\to M_n(K)$, where $M_n(K)$ is the full $n\times n$ matrix algebra over $K$, are of the form $\vf(X)=PXQ$ or $\vf(X)=PX^tQ$, where $P$ and $Q$ are fixed invertible matrices. In particular, $\vf$ is unital if and only if it is either an automorphism of $M_n(K)$ or an anti-automorphism of $M_n(K)$ (so that $\vf$ is a Jordan automorphism of $M_n(K)$). Marcus and Purves proved in~\cite[Theorem~2.1]{Marcus-Purves59} that the Dieudonn\'e's result holds over $K=\CC$ if one drops the bijectivity and strongness assumptions on $\vf$. Bre\v{s}ar and \v{S}emrl gave in~\cite{BresarSemrl99} an alternative proof of this fact by showing that unital linear invertibility preservers of $M_n(\CC)$ preserve idempotents. However, if $K$ is not algebraically closed, one cannot drop the strongness as observed in~\cite{Botta78}. The complete characterization of linear invertibility preservers of $M_n(K)$ over an arbitrary field $K$ has been accomplished relatively recently in~\cite{SeguinsPazzis10}. Soon after, \v{S}emrl~\cite{Semrl14} solved the problem for matrix rings over division rings. A non-linear generalization of~\cite[Theorem~3]{Dieudonne49} (in characteristic zero) has been obtained in~\cite{FosnerSemrl05}.
	
	On the other hand, the invertibility preservers of the algebra $T_n(K)$ of upper triangular $n\times n$ matrices over a field $K$ are in general far from being multiples of Jordan homomorphisms even in the bijective and strong case. More precisely, Chooi and Lim proved in~\cite[Theorem 3.3]{ChooiLim98} that $\vf:T_n(K)\to T_n(K)$ (with $|K|>2$) is a strong linear invertibility preserver if and only if, up to the multiplication by a fixed invertible matrix, $\vf$ permutes the diagonal entries of $X\in T_n(K)$. This characterization generalizes to the algebra $T_\infty(K)$ of infinite upper triangular matrices over $K$ whose entries are indexed by pairs of natural numbers, as claimed in~\cite[Theorem~1.1]{Slowik15} (see also~\cite[Theorem~3.3]{Slowik15} for the non-strong case and~\cite[Theorem~1.2]{Slowik15} that describes those invertibility preservers which preserve the inverses). However, the statements of~\cite[Theorems~1.1 and~3.3]{Slowik15} have some problems whenever $|K|<\infty$ (see \cref{connection-with-Slowik} below). 
	
	There are also numerous generalizations of the Dieudonn\'e's result to the operator algebras (see~\cite{JafarianSourour86,AupetitMouton94,BresarSemrl98,Sourour96,Aupetit2000,BresarFosnerSemrl03}).
	
	In this paper we describe linear unital invertibility preservers of the finitary incidence algebra $FI(X,K)$ of a poset $X$ over a field $K$. We mostly deal with the case $|K|>2$, but the case $|K|=2$ will also be treated at the end of the work.
	
	In \cref{sec-prelim} we recall the definition of $FI(X,K)$ and some facts on its structure that will be used throughout the text. As it will be seen below, the invertibility preservers of $FI(X,K)$ are closely related to certain maps on the power set algebra $\cP(X)$, so we also give some basic definitions from the theory of Boolean algebras and prove an auxiliary lemma. 
	
	\cref{sec-inv-pres-general} is the main part of the paper. In \cref{sec-inv-pres-|K|>2} we characterize in \cref{inv-pres-for-|K|>2} linear unital invertibility preservers of $FI(X,K)$ with $|K|>2$ in terms of $|K|$-complete endomorphisms of $\cP(X)$ and linear maps $FI(X,K)\to J(FI(X,K))$. Observe that in \cref{inv-pres-for-|K|>2} we do not assume bijectivity or strongness of the preservers --- these properties are then discussed in \cref{sec-bij-and-str} (see \cref{vf-strong<=>lb(A)-nonempty,bij-strong-|K|>2,surj-strong-|K|>2,bijective-is-strong}). We show in \cref{sec-connection} how \cref{inv-pres-for-|K|>2} applies to $T_n(K)$ and $T_\infty(K)$ and recover in \cref{inv-pres-|X|<infty,connection-with-Slowik} some results from \cite{ChooiLim98,Slowik15}. The invertibility preservers of $FI(X,K)$ that preserve the inverses are studied in \cref{sec-inverse-pres}. It turns out by \cref{vf-pres-inverses=>vf(1)vf-pres-idemp} that they are related to idempotent preservers, so that, whenever $|X|<\infty$ and $\ch(K)\ne 2$, the unital ones are exactly the unital Jordan endomorphisms of $FI(X,K)$ (see \cref{vf-pres-inverses=>vf-Jordan-homo,vf-pres-inverses=>vf-pm-auto-or-anti-auto}). We finish the paper providing in \cref{sec-inv-pres-over-Z_2} a characterization of linear unital invertibility preservers of $FI(X,\Z_2)$, which is similar to that of \cref{inv-pres-for-|K|>2} but involves a wider class of maps on $\cP(X)$ (see \cref{inv-pres-over-Z_2,vf-strong<=>lb-injective,bij-strong-over-Z_2}).
	
	\section{Preliminaries}\label{sec-prelim}
	
	\subsection{Invertibility preservers}
	
	If $\vf:A\to B$ is a linear invertibility preserver, then $\psi=\vf(1_A)\m \vf$ is a \textit{unital} linear invertibility preserver. Moreover, $\vf$ is strong (or bijective) if and only if $\psi$ is. Thus, the description of all/strong/bijective linear invertibility preservers $\vf:A\to B$ can be obtained from the description of the unital ones by multiplying by a fixed element of $U(B)$.
	
	\subsection{Finitary incidence algebras}
	Let $(X,\le)$ be a poset and $K$ a field.
	The \textit{finitary incidence algebra}~\cite{Khripchenko-Novikov09} of $X$ over $K$ is the $K$-space $FI(X,K)$ of formal sums of the form
	\begin{align}\label{af=formal-sum}
		\af = \sum_{x\leq y} \af_{xy}e_{xy},
	\end{align}
	where $x,y\in X$, $\af_{xy}\in K$ and $e_{xy}$ is a symbol, such that for all $x<y$ the set of $x \le u<v \le y$ with $\af_{uv}\ne 0$ is finite. The product in
	$FI(X,K)$ is given by the convolution
	\begin{align}\label{conv-product}
		\alpha\beta = \sum_{x\le y}\left(\sum_{x\le z \le y}\af_{xz}\bt_{zy} \right)e_{xy},
	\end{align}
	so that $FI(X,K)$ is associative and $\dl:=\sum_{x\in X}e_{xx}$ is its identity element. If $X$ is locally finite, $FI(X,K)$ coincides with the classical incidence algebra $I(X,K)$. In particular, for $X=\{1,\dots,n\}$ with the usual order, $FI(X,K)\cong T_n(K)$, the algebra of upper triangular $n\times n$ matrices over $K$, and for $X=\NN$ with the usual order, $FI(X,K)\cong T_\infty(K)$, the algebra of infinite upper triangular matrices over $K$, studied in~\cite{Slowik15}.
	
	We generalize the notation~\cref{af=formal-sum} as follows. Given $\af\in FI(X,K)$ and $\{\af_s\}_{s\in S}\sst FI(X,K)$, we write $\af=\sum_{s\in S}\af_s$ whenever for any $x\le y$ the set $S(x,y)=\{s\in S\mid (\af_s)_{xy}\ne 0\}$ is finite\footnote{In the most cases that we are working with $|S(x,y)|\le 1$.} and $\af_{xy}=\sum_{s\in S(x,y)}(\af_s)_{xy}$. As usually, we denote $e_x:=e_{xx}$ and $e_A:=\sum_{x\in A}e_x$ for all $A\sst X$. Observe that $e_X=\dl$.
	
	Given $\af\in FI(X,K)$, denote
	\begin{align*}
		\af_D=\sum_{x\in X}\af_{xx}e_{xx}\text{ and }\af_J=\sum_{x<y}\af_{xy}e_{xy},
	\end{align*}
	so that
	\begin{align}\label{af=af_D+af_J}
		\af=\af_D+\af_J.
	\end{align}
	Let 
	\begin{align*}
		D(X,K)&=\{\af\in FI(X,K)\mid \af_J=0\}\\
		J(FI(X,K))&=\{\af\in FI(X,K)\mid \af_D=0\}.
	\end{align*}
	Then $D(X,K)$ is a commutative subalgebra of $FI(X,K)$ and $J(FI(X,K))$ is an ideal of $FI(X,K)$ which coincides with the Jacobson radical of $FI(X,K)$ (see~\cite[Corollary 2]{Khripchenko-Novikov09}). There is an obvious isomorphism of $K$-spaces
	\begin{align*}
		FI(X,K)\cong D(X,K)\oplus J(FI(X,K))    
	\end{align*}
	induced by \cref{af=af_D+af_J}.
	
	The invertible elements of $FI(X,K)$ have been described in~\cite[Theorem 2]{Khripchenko-Novikov09}. Recall that 
	\begin{align*}
		\af\in U(FI(X,K))\iff \forall x\in X:\ \af_{xx}\ne 0.
	\end{align*} 
	In particular, $\af\in U(FI(X,K))\iff \af_D\in U(FI(X,K))$.
	
	\subsection{Boolean algebras}
	We follow here the definitions and terminology of~\cite{Monk}. A \textit{Boolean algebra} is a sextuple $(A,+,\cdot,-,0,1)$, where $A$ is a set, $+$ and $\cdot$ are binary operations on $A$, $-$ is a unary operation on $A$ and $0,1$ are distinguished elements of $A$, such that
	\begin{enumerate}
		\item $+$ and $\cdot$ are commutative and associative;
		\item $+$ and $\cdot$ are distributive with respect to each other;
		\item the absorption law holds: $x+(x\cdot y)=x$ and $x\cdot(x+y)=x$;
		\item the complementation law holds: $x+(-x)=1$ and $x\cdot(-x)=0$.
	\end{enumerate}
	A classical example is \textit{the power set algebra} $\cP(X)$, i.e. the Boolean algebra of all the subsets of a set $X$ under the usual set-theoretic operations of union, intersection and complement.
	
	For any Boolean algebra $A$ the relation 
	\begin{align*}
		x\le y\iff x+y=y\iff x\cdot y=x
	\end{align*}
	is a partial order on $A$. Moreover, $(A,\le)$ is a distributive complemented lattice with $0$ and $1$, where $x\vee y=x+y$ and $x\wedge y=x\cdot y$. A minimal (under $\le$) element of $A\setminus\{0\}$ (if it exists) is called an \textit{atom} of $A$. In $\cP(X)$ we have $U\le V\iff U\sst V$, so the atoms of $\cP(X)$ are exactly $\{x\}$, where $x\in X$.
	
	For any $\{a_i\}_{i\in I}\sst A$ denote 
	\begin{align*}
		\sum_{i\in I} a_i:=\sup_{i\in I} a_i\text{ and }\prod_{i\in I} a_i:=\inf_{i\in I} a_i\ (\text{under } \le),  
	\end{align*}
	if they exist. Given a cardinal $\kp$, we say that $A$ is \textit{$\kp$-complete} if $\sum_{i\in I} a_i$ and $\prod_{i\in I} a_i$ exist for all $\{a_i\}_{i\in I}\sst A$ with $|I|\le \kp$. Moreover, $A$ is \textit{complete} if it is $\kp$-complete for all $\kp$. Clearly, the power set algebra $\cP(X)$ is complete.
	
	Homomorphisms, endomorphisms, isomorphisms and automorphisms of Boolean algebras are defined in a natural way (observe that all of them are required to preserve $0$ and $1$). A homomorphism $f:A\to B$ of Boolean algebras is \textit{$\kp$-complete} ($\kp$ is a cardinal) if for all $\{a_i\}_{i\in I}\sst A$ with $|I|\le \kp$:
	\begin{align*}
		\exists\sum_{i\in I} a_i\ \impl\ \exists\sum_{i\in I} f(a_i)=f\left(\sum_{i\in I} a_i\right).    
	\end{align*}
	
	\begin{exm}\label{non-complete-endo-P(X)}
		Let $|X|=\kp$, an infinite cardinal. Fix $Y\subsetneq X$ with $|Y|=\kp$ and a bijection $\mu:X\to Y$. For any finite $A\sst X$ define $\lb(A)=\{\mu(a)\mid a\in A\}$ and $\lb(X\setminus A)=X\setminus\lb(A)$. Then $\lb:\mathcal{FC}(X)\to\cP(X)$ is a homomorphism, where $\mathcal{FC}(X)$ is the Boolean subalgebra of $\cP(X)$ consisting of finite or cofinite $A\in\cP(X)$. Since $\cP(X)$ is complete, by Sikorski's extension theorem (\cite[Theorem 5.9]{Monk}), $\lb$ extends to an endomorphism of $\cP(X)$. But $\bigcup_{x\in X}\lb(\{x\})=\bigcup_{x\in X}\{\mu(x)\}=Y\ne X=\lb(X)$, so the extended $\lb$ is not $\kp$-complete.
	\end{exm}
	
	It turns out that in some cases for $\kp$-completeness of $f:A\to B$ it suffices to deal with more specific families of elements of $A$. We say that $a,b\in A$ are \textit{disjoint} if $a\cdot b=0$. A family $\{a_i\}_{i\in I}\sst A$ of pairwise disjoint elements of $A$ is a \textit{partition of $1$} if $\sum_{i\in I}a_i=1$ (observe that we do not require that $a_i>0$ as in \cite[Definition~3.3]{Monk}).

	\begin{lem}\label{f-complete-iff-preserves-partitions}
		Let $A$ and $B$ be $\kp$-complete Boolean algebras. Then a homomorphism $f:A\to B$ is $\kp$-complete if and only if $\sum_{i\in I}f(a_i)=1$ whenever $\{a_i\}_{i\in I}\sst A$ is a partition of $1$ and $|I|\le\kp$. 
	\end{lem}
	\begin{proof}
		
		
		We only need to prove the ``if'' part. The statement is trivial for any finite $I$, so take an arbitrary infinite $\{a_i\}_{i\in I}\sst A$ with $|I|\le\kp$ and denote $a=\sum_{i\in I}a_i$. By \cite[Lemma 3.12]{Monk} there is a family $\{b_i\}_{i\in I}\sst A$ of pairwise disjoint elements of $A$ such that $b_i\le a_i$ for all $i\in I$ and $\sum_{i\in I}b_i=a$. Then $f(b_i)\le f(a_i)\le \sum_{i\in I}f(a_i)$ implies $\sum_{i\in I}f(b_i)\le \sum_{i\in I}f(a_i)$. Furthermore, $a_i\le a$ yields $f(a_i)\le f(a)$, whence $\sum_{i\in I}f(a_i)\le f(a)$. Thus, we have
		\begin{align*}
			\sum_{i\in I}f(b_i)\le \sum_{i\in I}f(a_i)\le f(a).
		\end{align*}
		To complete the proof, it suffices to show that $f(a)\le\sum_{i\in I}f(b_i)$. If $a=1$, then $\{b_i\}_{i\in I}$ is a partition of $1$, and the claim holds by the assumption on $f$. Otherwise, $\{b_i\}_{i\in I}\cup\{-a\}$ is a partition of $1$ (observe that $b_i\cdot (-a)\le a\cdot (-a)=0$) of the same cardinality (recall that $I$ is infinite). Then $\sum_{i\in I}f(b_i)+(-f(a))=1$. Multiplying both sides by $f(a)$, we have $\left(\sum_{i\in I}f(b_i)\right)\cdot f(a)=f(a)$, i.e. $f(a)\le \sum_{i\in I}f(b_i)$, as needed. 
	\end{proof}
	
	\section{Invertibility preservers of finitary incidence algebras}\label{sec-inv-pres-general}
	
	We begin with some lemmas that hold in the general case.	
	
	\begin{lem}\label{vf-maps-J-to-J}
		Let $\vf:FI(X,K)\to FI(X,K)$ be a unital linear invertibility preserver. Then $\vf(J(FI(X,K)))\sst J(FI(X,K))$.
	\end{lem}
	\begin{proof}
		Let $\af\in J(FI(X,K))$. Assume that $\vf(\af)\not\in J(FI(X,K))$, i.e. $\vf(\af)_{xx}\ne 0$ for some $x\in X$. Define 
		\begin{align*}
			k:=-(\vf(\af)_{xx})\m \text{ and }\bt:=\dl+k\af.
		\end{align*}
		Then $\bt\in U(FI(X,K))$, whence $\vf(\bt)=\dl+k\vf(\af)\in U(FI(X,K))$. But $\vf(\bt)_{xx}=1+k\vf(\af)_{xx}=0$, a contradiction.
	\end{proof}
	
	\begin{cor}\label{vf(f)_D-is-vf(f_D)_D}
		Let $\vf:FI(X,K)\to FI(X,K)$ be a unital linear invertibility preserver. Then for any $\af\in FI(X,K)$:
		\begin{align*}
			\vf(\af)_D=\vf(\af_D)_D.
		\end{align*}
	\end{cor}
	\begin{proof}
		We have $\vf(\af)_D=\vf(\af_D+\af_J)_D=\vf(\af_D)_D+\vf(\af_J)_D=\vf(\af_D)_D$, where the latter equality is due to \cref{vf-maps-J-to-J}.
	\end{proof}
	
	\begin{lem}\label{from-vf-to-lb}
		Let $\vf:FI(X,K)\to FI(X,K)$ be a unital linear invertibility preserver. Then there exists $\lb:\cP(X)\to \cP(X)$ such that for all $A\sst X$:
		\begin{align}\label{vf(e_A)_D=e_lb(A)}
			\vf(e_A)_D=e_{\lb(A)}.
		\end{align}
	\end{lem}
	\begin{proof}
		Take an arbitrary $A\sst X$ and assume that $\vf(e_A)_{xx}\not\in\{0,1\}$ for some $x\in X$. Then define 
		\begin{align*}
			k:=1-(\vf(e_A)_{xx})\m\ne 0\text{ and }\af:=\dl+(k-1)e_A.
		\end{align*}
		Since 
		\begin{align*}
			\af_{yy}=
			\begin{cases}
				k, & y\in A,\\
				1, & y\not\in A,
			\end{cases}
		\end{align*}
		we have $\af\in U(FI(X,K))$. Hence, $\vf(\af)=\dl+(k-1)\vf(e_A)\in U(FI(X,K))$. However, $\vf(\af)_{xx}=1+(k-1)\vf(e_A)_{xx}=0$, a contradiction. Thus, 
		\begin{align*}
			\lb(A):=\{x\in X\mid \vf(e_A)_{xx}=1\}
		\end{align*}
		defines the desired map $\lb:\cP(X)\to \cP(X)$ satisfying \cref{vf(e_A)_D=e_lb(A)}.
	\end{proof}
	
	\subsection{The case $|K|>2$}\label{sec-inv-pres-|K|>2}
	In order to proceed, we have to impose certain restrictions on $K$. We first consider the case $|K|>2$. 
	\begin{lem}\label{lb-separating}
		Let $|K|>2$, $\vf:FI(X,K)\to FI(X,K)$ be a unital linear in\-ver\-ti\-bi\-li\-ty preserver and $\lb:\cP(X)\to \cP(X)$ the associated map as in \cref{from-vf-to-lb}. Then for all $A,B\in\cP(X)$:
		\begin{align}\label{A-cap-B=empty=>lb(A)-cap-lb(B)=empty}
			A\cap B=\emptyset\impl\lb(A)\cap\lb(B)=\emptyset.
		\end{align}
	\end{lem}
	\begin{proof}
		Assume that $A\cap B=\emptyset$, but $\lb(A)\cap\lb(B)\ne\emptyset$. Choose arbitrary $x\in \lb(A)\cap\lb(B)$ and $k\in K\setminus\{0,1\}$ and define
		\begin{align*}
			\af:=\dl+(k-1)e_A-ke_B.
		\end{align*}
		Then 
		\begin{align*}
			\af_{yy}=
			\begin{cases}
				k, & y\in A,\\
				1-k, & y\in B,\\
				1, & y\not\in A\sqcup B,
			\end{cases}
		\end{align*}
		so $\af\in U(FI(X,K))$. Therefore, $\vf(\af)=\dl+(k-1)e_{\lb(A)}-ke_{\lb(B)}\in U(FI(X,K))$. However, $\vf(\af)_{xx}=1+(k-1)-k=0$, a contradiction.
	\end{proof}
	
	\begin{defn}
		A map $\lb:\cP(X)\to \cP(X)$ satisfying \cref{A-cap-B=empty=>lb(A)-cap-lb(B)=empty} will be called \textit{separating}.
	\end{defn}
	
	
	\begin{exm}[Remark 3.2 from \cite{ChooiLim98}]
		Let $K=\Z_2$ and $|X|>2$. Fix distinct $x,y,z\in X$. Then $\vf(\af)=(\af_{xx}+\af_{yy}+\af_{zz})e_x+\sum_{w\ne x}\af_{ww}e_w$ is a unital linear strong invertibility preserver of $FI(X,K)$, whose $\lb:\cP(X)\to \cP(X)$ is not separating. 
		
		Indeed, the fact that $\vf$ is a unital strong invertibility preserver is due to
		\begin{align*}
			\af\in U(FI(X,\Z_2))\iff \af_D=\dl\iff \vf(\af)=\dl\iff \vf(\af)\in U(FI(X,\Z_2)).   
		\end{align*}
		Since $\vf(e_x)=e_x$ and $\vf(e_y)=e_x+e_y$, we have $\lb(\{x\})=\{x\}$ and $\lb(\{y\})=\{x,y\}$, so that $\lb$ is obviously not separating.
	\end{exm}

	\begin{lem}\label{lb-preserves-diff-and-cap}
		Let $|K|>2$, $\vf:FI(X,K)\to FI(X,K)$ be a unital linear invertibility preserver and $\lb:\cP(X)\to \cP(X)$ the associated map as in \cref{from-vf-to-lb}. Then $\lb$ is an endomorphism of the Boolean algebra $\cP(X)$.
	\end{lem}
	\begin{proof}
		Since $\dl=e_X$, we have 
		\begin{align}\label{lb(X)-is-X}
			\lb(X)=X.
		\end{align} 
		Furthermore, it follows from $e_{\lb(X\setminus A)}=\vf(e_{X\setminus A})=\vf(\dl-e_A)=\dl-e_{\lb(A)}=e_{X\setminus\lb(A)}$ that
		\begin{align}\label{lb(X-min-A)=X-min-lb(A)}
			\lb(X\setminus A)=X\setminus\lb(A).
		\end{align}
		Let us prove 
		\begin{align}\label{lb(A-cap-B)=lb(A)-cap-lb(B)}
			\lb(A\cap B)=\lb(A)\cap\lb(B).
		\end{align}
		
		\textit{Case 1.} $\ch(K)\ne 2$. Observe that $e_A+e_B=e_{A\tr B}+2e_{A\cap B}$, so $\vf(e_A+e_B)=e_{\lb(A\tr B)}+2e_{\lb(A\cap B)}$. Then
		\begin{align}\label{lb(A-cap-B)=L_2(vf(e_A+e_B))}
			\lb(A\cap B)=\{x\in X\mid \vf(e_A+e_B)_{xx}=2\},
		\end{align}
		because $\lb$ is separating by \cref{lb-separating} and $0\ne 2$. On the other hand, $\vf(e_A+e_B)=\vf(e_A)+\vf(e_B)=e_{\lb(A)}+e_{\lb(B)}=e_{\lb(A)\tr \lb(B)}+2e_{\lb(A)\cap \lb(B)}$, whence
		\begin{align}\label{lb(A)-cap-lb(B)=L_2(vf(e_A+e_B))}
			\lb(A)\cap \lb(B)=\{x\in X\mid \vf(e_A+e_B)_{xx}=2\}.
		\end{align}
		Combining \cref{lb(A-cap-B)=L_2(vf(e_A+e_B)),lb(A)-cap-lb(B)=L_2(vf(e_A+e_B))}, we get \cref{lb(A-cap-B)=lb(A)-cap-lb(B)}.
		
		\textit{Case 2.} $\ch(K)=2$. Choose $k\in K\setminus\{0,1\}$ and define $\af=ke_A+(k+1)e_B$. Then $\af=ke_{A\setminus B}+(k+1)e_{B\setminus A}+e_{A\cap B}$. Applying $\vf$, we get $\vf(\af)=ke_{\lb(A\setminus B)}+(k+1)e_{\lb(B\setminus A)}+e_{\lb(A\cap B)}$. Since $\lb$ is separating by \cref{lb-separating} and $0$, $1$, $k$ and $k+1$ are all distinct, we have $\lb(A\cap B)=\{x\in X\mid \vf(\af)_{xx}=1\}$. On the other hand, $\vf(\af)=ke_{\lb(A)}+(k+1)e_{\lb(B)}=ke_{\lb(A)\setminus \lb(B)}+(k+1)e_{\lb(B)\setminus \lb(A)}+e_{\lb(A)\cap \lb(B)}$, whence $\lb(A)\cap \lb(B)=\{x\in X\mid \vf(\af)_{xx}=1\}$. Thus, \cref{lb(A-cap-B)=lb(A)-cap-lb(B)} holds.
		
		Since $\emptyset = X\setminus X$, we obtain $\lb(\emptyset)=\emptyset$ from \cref{lb(X)-is-X,lb(X-min-A)=X-min-lb(A)}. The De Morgan's laws together with \cref{lb(X-min-A)=X-min-lb(A),lb(A-cap-B)=lb(A)-cap-lb(B)} imply that $\lb(A\cup B)=\lb(A)\cup\lb(B)$.
	\end{proof}
	
	
	
	

	\begin{lem}\label{union-lb(L_k(f))=X}
		Let $\vf:FI(X,K)\to FI(X,K)$ be a unital linear invertibility preserver and $\lb:\cP(X)\to \cP(X)$ the associated map as in \cref{from-vf-to-lb}. Then for any partition $\{X_i\}_{i\in I}$ of $X$ of cardinality at most $|K|$ we have
		\begin{align}\label{union-lb(X_k)=X}
			\bigcup_{i\in I}\lb(X_i)=X.
		\end{align}
	\end{lem}
	\begin{proof}
		Assume that there is $x\in X$ such that $x\not\in\lb(X_i)$ for all $i\in I$. Choose an arbitrary subset $\{k_i\}_{i\in I}\sst K$ of cardinality $|I|$ with $k_i\ne k_j$ ($i\ne j$) and define $\af=\sum_{i\in I}k_ie_{X_i}$. Denote $k:=\vf(\af)_{xx}$.
		
		\textit{Case 1.} There exists (a unique) $i\in I$ such that $k_i=k$. Then define $\bt:=\af-k\dl+e_{X_i}$.
		We have
		\begin{align*}
			\bt_{yy}=
			\begin{cases}
				1, & y\in X_i,\\
				\af_{yy}-k, & y\not\in X_i,
			\end{cases}
		\end{align*}
		where $\af_{yy}-k\ne 0$ for $y\not\in X_i$ because $k\ne k_j$ for $j\ne i$. Hence, $\bt\in U(FI(X,K))$, and consequently $\vf(\bt)\in U(FI(X,K))$. However, $\vf(\bt)_{xx}=\vf(\af)_{xx}-k+(e_{\lb(X_i)})_{xx}=k-k+0=0$, a contradiction.
		
		\textit{Case 2.} $k_i\ne k$ for all $i\in I$. Then define $\bt:=\af-k\dl$. Since $\bt_{yy}=\af_{yy}-k\ne 0$ for all $y\in X$, we have $\bt\in U(FI(X,K))$, and consequently $\vf(\bt)\in U(FI(X,K))$. However, $\vf(\bt)_{xx}=\vf(\af)_{xx}-k=k-k=0$, a contradiction.
	\end{proof}
	
	\begin{cor}
		Let $|K|>2$, $\vf:FI(X,K)\to FI(X,K)$ be a unital linear invertibility preserver and $\lb:\cP(X)\to \cP(X)$ the associated map as in \cref{from-vf-to-lb}. Then $\lb$ is a $|K|$-complete endomorphism of the Boolean algebra $\cP(X)$.
	\end{cor}
	\begin{proof}
		This is a consequence of \cref{lb-preserves-diff-and-cap,f-complete-iff-preserves-partitions,union-lb(L_k(f))=X}.
	\end{proof}
	
	Given $\af\in FI(X,K)$ and $k\in K$, denote 
	\begin{align*}
		L_k(\af):=\{x\in X\mid \af_{xx}=k\}.
	\end{align*}
	Clearly, $\{L_k(\af)\}_{k\in K}$ is a partition of $X$. Observe that $\af_D=\sum_{k\in K}ke_{L_k(\af)}$ for any $\af\in FI(X,K)$.

	\begin{lem}\label{vf(f)_D=sum-k-e_lb(L_k)}
		Let $|K|>2$, $\vf:FI(X,K)\to FI(X,K)$ be a unital linear invertibility preserver and $\lb:\cP(X)\to \cP(X)$ the associated map as in \cref{from-vf-to-lb}. Then for all $\af\in FI(X,K)$:
		\begin{align}\label{vf(f)_D=sum-ke_lb(L_k(f))}
			\vf(\af)_D=\sum_{k\in K}ke_{\lb(L_k(\af))}.
		\end{align}
	\end{lem}
	\begin{proof}
		Given $x\in X$, by \cref{union-lb(L_k(f))=X} there is $k\in K$ such that $x\in\lb(L_k(\af))$. Since $\lb$ is separating by \cref{lb-separating}, such an element $k$ is unique. In order to establish \cref{vf(f)_D=sum-ke_lb(L_k(f))}, we have to prove that $\vf(\af)_{xx}=k$. Assume that 
		\begin{align}\label{vf(f)(y_y)=l-for-y-in-lb(L_k(f))}
			\vf(\af)_{xx}=l\ne k. 
		\end{align}
		
		\textit{Case 1.} $k\ne 0$. Define $\bt:=\af-le_{L_k(\af)}+e_{L_0(\af)}$. Then
		\begin{align*}
			\bt_{yy}=
			\begin{cases}
				1, & y\in L_0(\af),\\
				k-l, & y\in L_k(\af),\\
				\af_{yy}, & y\not\in L_0(\af)\sqcup L_k(\af),
			\end{cases} 
		\end{align*} 
		so $\bt\in U(FI(X,K))$ and thus $\vf(\bt)\in U(FI(X,K))$. However $\vf(\bt)_{xx}=\vf(\af)_{xx}-l(e_{\lb(L_k(\af))})_{xx}+(e_{\lb(L_0(\af))})_{xx}=l-l=0$ by \cref{vf(f)(y_y)=l-for-y-in-lb(L_k(f))} and the fact that $\lb(L_k(\af))\cap\lb(L_0(\af))=\emptyset$.
		
		\textit{Case 2.} $k=0$, so that $l\ne 0$ and $x\in\lb(L_0(\af))$. Define $\bt:=\af-le_{L_0(\af)}$. Then
		\begin{align*}
			\bt_{yy}=
			\begin{cases}
				-l, & y\in L_0(\af),\\
				\af_{yy}, & y\not\in L_0(\af).
			\end{cases} 
		\end{align*} 
		Hence, $\bt\in U(FI(X,K))$, and consequently $\vf(\bt)\in U(FI(X,K))$. But $\vf(\bt)_{xx}=\vf(\af)_{xx}-l(e_{\lb(L_0(\af))})_{xx}=l-l=0$ by \cref{vf(f)(y_y)=l-for-y-in-lb(L_k(f))}.
		
		Thus, both cases lead to a contradiction.
	\end{proof}
	
	%
	%
	
	\begin{lem}\label{from-psi-and-lb-to-vf}
		Any $|K|$-complete endomorphism $\lb$ of the Boolean algebra $\cP(X)$ and a linear map $\psi:FI(X,K)\to J(FI(X,K))$ define a linear invertibility preserver $\vf:FI(X,K)\to FI(X,K)$ by means of
		\begin{align}\label{vf(f)=sum-f(x_x)e_lb(x)+psi(f_J)}
			\vf(\af)=\sum_{k\in K}ke_{\lb(L_k(\af))}+\psi(\af).
		\end{align}
		Moreover, $\vf$ is unital if and only if 
		\begin{align}\label{psi(dl)=0}
			\psi(\dl)=0.
		\end{align}
	\end{lem}
	\begin{proof}
		Since the sets $\{L_k(\af)\}_{k\in K}$ are pairwise disjoint and $\lb$ is an endomorphism of $\cP(X)$, then the sets $\{\lb(L_k(\af))\}_{k\in K}$ are also pairwise disjoint, so $\sum_{k\in K}ke_{\lb(L_k(\af))}$ makes sense and belongs to $D(X,K)$. It is hence clear that $\vf(\af)\in FI(X,K)$, because $\psi(\af)\in FI(X,K)$. Thus, $\vf$ is well-defined. 
		
		Let us prove that $\vf$ is linear. If $x<y$, then $\vf(\af)_{xy}=\psi(\af)_{xy}$, which is linear in $\af$. Now let $\af,\bt\in FI(X,K)$ and $x\in X$. Since $\{L_k(\af)\}_{k\in K}$ and $\{L_k(\bt)\}_{k\in K}$ are partitions of $X$ of cardinality $|K|$ and $\lb$ is $|K|$-complete, then $\{\lb(L_k(\af))\}_{k\in K}$ and $\{\lb(L_k(\bt))\}_{k\in K}$ are also partitions of $X$. Hence, there are unique $k,l\in K$ such that $x\in \lb(L_k(\af))\cap \lb(L_l(\bt))$. Then $\vf(\af)_{xx}=k$ and $\vf(\bt)_{xx}=l$ by \cref{vf(f)=sum-f(x_x)e_lb(x)+psi(f_J)}. Now, $L_k(\af)\cap L_l(\bt)\sst L_{k+l}(\af+\bt)$, whence 
		\begin{align*}
			\lb(L_k(\af))\cap \lb(L_l(\bt))=\lb(L_k(\af)\cap L_l(\bt))\sst\lb(L_{k+l}(\af+\bt)),
		\end{align*}
		because $\lb$ is an endomorphism of $\cP(X)$. It follows that $x\in \lb(L_{k+l}(\af+\bt))$, whence $\vf(\af+\bt)_{xx}=k+l=\vf(\af)_{xx}+\vf(\bt)_{xx}$. It remains to show that $\vf(k\af)_{xx}=k\vf(\af)_{xx}$ for all $k\in K$. This is obvious for $k=0$. And if $k\ne 0$, then $L_l(k\af)=L_{k\m l}(\af)$, so, whenever $x\in \lb(L_l(k\af))$, we have $\vf(k\af)_{xx}=l=k\cdot(k\m l)=k\vf(\af)_{xx}$. 
		
		Let $\af\in U(FI(X,K))$, so that $L_0(\af)=\emptyset$. Then $\lb(L_0(\af))=\emptyset$, whence $\bigsqcup_{k\ne 0}\lb(L_k(\af))=X$, showing that $\vf(\af)_{xx}\ne 0$ for all $x\in X$ by \cref{vf(f)=sum-f(x_x)e_lb(x)+psi(f_J)}. Thus, $\vf(\af)\in U(FI(X,K))$. 
		
		Clearly, $\vf(\dl)_D=\dl$ due to \cref{vf(f)=sum-f(x_x)e_lb(x)+psi(f_J)}, because $L_1(\dl)=X$ and $\lb(X)=X$. Thus, $\vf(\dl)=\dl$ is equivalent to \cref{psi(dl)=0}.
	\end{proof}
	
	\begin{thrm}\label{inv-pres-for-|K|>2}
		Let $|K|>2$. Then unital linear invertibility preservers of $FI(X,K)$ are exactly the maps of the form \cref{vf(f)=sum-f(x_x)e_lb(x)+psi(f_J)}, where $\lb$ is a $|K|$-complete endomorphism of the Boolean algebra $\cP(X)$ and $\psi:FI(X,K)\to J(FI(X,K))$ is a linear map annihilating $\dl$.
	\end{thrm}
	\begin{proof}
		We know from \cref{from-psi-and-lb-to-vf} that any map of the form \cref{vf(f)=sum-f(x_x)e_lb(x)+psi(f_J)} is a unital linear invertibility preserver of $FI(X,K)$. Conversely, let $\vf$ be a unital linear invertibility preserver of $FI(X,K)$. Then $\vf(\af)=\vf(\af)_D+\vf(\af)_J$, where $\vf(\af)_D$ is given by \cref{vf(f)_D=sum-ke_lb(L_k(f))} in view of \cref{vf(f)_D=sum-k-e_lb(L_k)}. It remains to define $\psi(\af)=\vf(\af)_J$ to get the desired form \cref{vf(f)=sum-f(x_x)e_lb(x)+psi(f_J)}.
	\end{proof}
	
	\subsection{Strongness and bijectivity of an invertibility preserver}\label{sec-bij-and-str}
	
	We now specify the result of \cref{inv-pres-for-|K|>2} to the case where $\vf$ is strong and/or bijective.
	\begin{cor}\label{vf-strong<=>lb(A)-nonempty}
		Under the conditions of \cref{inv-pres-for-|K|>2}, the unital invertibility preserver $\vf$ is strong if and only if the endomorphism $\lb$ is injective. 
	\end{cor}
	\begin{proof}
		Observe by \cite[Lemma 5.3]{Monk} that $\lb$ is injective if and only if for any $A\in \cP(X)$:
		\begin{align}\label{lb(A)-empty=>A-empty}
			\lb(A)=\emptyset\impl A=\emptyset.
		\end{align}
		
		Let $\vf$ be strong. Assume that $\lb(A)=\emptyset$ for some $A\sst X$. Then $\vf(e_A)\in J(FI(X,K))$, so $\vf(\dl-e_A)=\dl-\vf(e_A)\in U(FI(X,K))$. Since $\vf$ is strong, $\dl-e_A\in U(FI(X,K))$, whence $A=\emptyset$.
		
		Conversely, assume \cref{lb(A)-empty=>A-empty}. If $\vf(\af)\in U(FI(X,K))$, then $\lb(L_0(\af))=\emptyset$ by \cref{vf(f)=sum-f(x_x)e_lb(x)+psi(f_J)}. Hence, $L_0(\af)=\emptyset$ thanks to \cref{lb(A)-empty=>A-empty}, i.e. $\af\in U(FI(X,K))$. Thus, $\vf$ is strong.
	\end{proof}
	
	The next easy example shows that a strong invertibility preserver of $FI(X,K)$ need not be bijective even if $X$ is finite.
	\begin{exm}
		For any $X$ and $K$ the map $\af\mapsto \af_D$ is a unital linear strong invertibility preserver $FI(X,K)\to FI(X,K)$ that is neither injective nor surjective unless $X$ is an anti-chain. The corresponding $\lb:\cP(X)\to\cP(X)$ is the identity map. 
	\end{exm}
	
	\begin{cor}\label{bij-strong-|K|>2}
		Under the conditions of \cref{inv-pres-for-|K|>2}, the unital linear invertibility preserver $\vf$ is bijective and strong if and only if
		\begin{enumerate}
			\item\label{lb-automorphism} the associated map $\lb$ is an automorphism of the Boolean algebra $\cP(X)$,
			\item\label{psi-bijection-of-J} $\psi$ maps bijectively $J(FI(X,K))$ onto $J(FI(X,K))$.
		\end{enumerate}
	\end{cor}
	\begin{proof}
		Let $\vf$ be bijective and strong. By \cref{vf-strong<=>lb(A)-nonempty} the map $\lb$ is injective. Given $B\in\cP(X)$, there exists $\af\in FI(X,K)$ such that $\vf(\af)=e_B$. By \cref{vf(f)=sum-f(x_x)e_lb(x)+psi(f_J)} we have $\af_D=e_A$ and $\lb(A)=B$ for some $A\in\cP(X)$. Thus, $\lb$ is surjective, giving \cref{lb-automorphism}. Since $\vf|_{J(FI(X,K))}=\psi|_{J(FI(X,K))}:J(FI(X,K))\to J(FI(X,K))$ by \cref{vf(f)=sum-f(x_x)e_lb(x)+psi(f_J)} and $\vf(J(FI(X,K)))=J(FI(X,K))$ by \cref{lb-automorphism,vf(f)=sum-f(x_x)e_lb(x)+psi(f_J)}, we get \cref{psi-bijection-of-J}. 
		
		Conversely, assume \cref{psi-bijection-of-J,lb-automorphism}. By \cref{vf-strong<=>lb(A)-nonempty} the preserver $\vf$ is strong. If $\vf(\af)=0$, then by \cref{lb-automorphism,vf(f)=sum-f(x_x)e_lb(x)+psi(f_J)} we see that $\af\in J(FI(X,K))$ and $\psi(\af)=0$. Hence, $\af=0$ by \cref{psi-bijection-of-J}. For the surjectivity of $\vf$, take $\bt\in FI(X,K)$. Then $\af\in FI(X,K)$ with $\af_D=\sum_{k\in K}ke_{\lb\m(L_k(\bt))}$ and $\af_J=\psi\m(\bt_J-\psi(\af_D))$ is such that $\vf(\af)=\bt$ by \cref{vf(f)=sum-f(x_x)e_lb(x)+psi(f_J)}.
	\end{proof}
	
	Analysing the proof of \cref{bij-strong-|K|>2}, one can observe the following.
	\begin{rem}\label{surj-strong-|K|>2}
		\cref{bij-strong-|K|>2} will still hold, if we replace ``$\vf$ is bijective'' by ``$\vf$ is surjective'' and \cref{psi-bijection-of-J} by the weaker condition $\psi(J(FI(X,K)))=J(FI(X,K))$.
	\end{rem}
	
	\begin{exm}\label{inj-non-surj-lb}
		Let $X=\NN$ with the usual order and $K$ be an arbitrary field. For any $\af\in FI(\NN,K)$ define 
		\begin{align*}
			\vf(\af)=\sum_{n\in\NN}\af_{nn}e_{\{2n,2n+1\}}+\af_J.
		\end{align*}
		It is clear that $\vf:FI(\NN,K)\to FI(\NN,K)$ is an injective non-surjective unital linear strong invertibility preserver. The corresponding $\lb:\cP(\NN)\to\cP(\NN)$ is non-surjective, as $|\lb(A)|=2|A|$ for all $A\in\cP(\NN)$.
		
		On the other hand, if we define
		\begin{align*}
			\vf(\af)=\af_D+\eta(\af_J),
		\end{align*}
		where $\eta:J(FI(\NN,K))\to J(FI(\NN,K))$ is any surjective non-injective linear map (such a map exists as $\dim_K(J(FI(\NN,K)))=\infty$), then $\vf$ is a surjective non-injective unital linear strong invertibility preserver. The corresponding $\lb:\cP(\NN)\to\cP(\NN)$ is the identity map.
	\end{exm}

	The next fact, which makes \cref{bij-strong-|K|>2}\cref{lb-automorphism} more precise, is probably well-known (and easy to prove), but we could not find a reference.
	\begin{rem}\label{auto-P(X)}
		The automorphisms of the Boolean algebra $\cP(X)$ are exactly the maps $\cP(X)\to \cP(X)$ of the form 
		\begin{align}\label{lb(A)=set-eta(x)-x-in-A}
			\lb(A)=\{\mu(x)\mid x\in A\},    
		\end{align}
		where $\mu$ is a bijection $X\to X$. In particular, they are automatically $\kp$-complete for any $\kp$.
		
		Indeed, it is obvious that any map of the form \cref{lb(A)=set-eta(x)-x-in-A} is an automorphism of $\cP(X)$. Conversely, any automorphism $\lb$ of $\cP(X)$ maps the atoms of $\cP(X)$ to the atoms of $\cP(X)$, which induces a bijection $\mu:X\to X$ such that $\lb(\{x\})=\{\mu(x)\}$ for all $x\in X$. Given $A\in \cP(X)$ and $x\in X$, we have $x\in A\iff \{x\}\sst A\iff \{\mu(x)\}\sst\lb(A)\iff \mu(x)\in \lb(A)$, showing \cref{lb(A)=set-eta(x)-x-in-A}.
	\end{rem}
	
	\begin{rem}\label{bijective-is-strong}
		If $|\{(x,y)\mid x<y\}|<\infty$, then any unital bijective linear invertibility preserver $\vf$ of $FI(X,K)$ is automatically strong.
		
		For, in this case $\dim_K(J(FI(X,K)))<\infty$, so $\vf(J(FI(X,K)))=J(FI(X,K))$ in view of \cref{vf-maps-J-to-J} and the bijectivity of $\vf$. Hence, $\vf(e_A)_D=e_{\lb(A)}\ne 0$ for any nonempty $A\in\cP(X)$, and thus $\lb(A)\ne\emptyset$, i.e. $\lb$ is injective.
	\end{rem}
	
	The following example shows that, whenever $|\{(x,y)\mid x<y\}|=\infty$, non-strong unital bijective invertibility preservers of $FI(X,K)$ can exist.
	\begin{exm}\label{exm-bijective-non-strong}
		Let $X=\NN$ with the usual order and $K$ be an arbitrary field. Extend $e_{01}\in J(FI(\NN,K))$ to a $K$-basis $B$ of $J(FI(\NN,K))$. Since $\dim_K(J(FI(\NN,K)))=\infty$, we have $|B|=|B\setminus\{e_{01}\}|$. Denote $C=\Span_K(B\setminus\{e_{01}\})$ and let $\eta:J(FI(\NN,K))\to C$ be an isomorphism of $K$-spaces. Define
		\begin{align}\label{vf(f)=sum-f(n_n)e_n-1+f(0_0)e_01+eta(f_J)}
			\vf(\af)=\sum_{n=1}^\infty \af_{nn}e_{n-1}+\af_{00}e_{01}+\eta(\af_J).
		\end{align}
		It is easy to see that $\vf$ is a linear invertibility preserver of $FI(\NN,K)$. 
		Moreover, $\vf$ is a bijection whose inverse is
		\begin{align*}
			\vf\m(\bt)=\sum_{n=1}^\infty \bt_{n-1,n-1}e_n+\bt_{01}e_0+\eta\m(\bt_J-\bt_{01}e_{01}).
		\end{align*}
		However, $\vf$ is not strong because $\vf(e_{\NN\setminus\{0\}})=\dl$. Then $\psi(\af)=\vf(\dl)\m\vf(\af)$ is a unital bijective linear invertibility preserver that is not strong.
	\end{exm}
	
	\begin{rem}
		\cref{exm-bijective-non-strong} obviously generalizes to any $X$ with $|\{(x,y)\mid x<y\}|=\infty$ (replace in \cref{vf(f)=sum-f(n_n)e_n-1+f(0_0)e_01+eta(f_J)} the elements $0$ and $1$ by a fixed pair $x_0<y_0$ from $X$ and $e_{n-1}$ by $e_{\mu(x)}$, where $\mu$ is a bijection $X\setminus\{x_0\}\to X$).
	\end{rem}
	
		
	
	\subsection{Connection with the known results}\label{sec-connection}
	
	The following remark establishes a connection between \cref{inv-pres-for-|K|>2} and \cite[Theorem 3.3]{ChooiLim98}.
	\begin{rem}\label{inv-pres-|X|<infty}
		Let $|X|<\infty$ and $K$ be an arbitrary field. Then any endomorphism $\lb$ of $\cP(X)$ is $|K|$-complete and is uniquely determined by the partition $\{\lb(\{x\})\}_{x\in X}$ of $X$ (and conversely: any partition $\{A_x\}_{x\in X}$ of $X$ of cardinality $|X|$ defines an endomorphism $\lb$ of $\cP(X)$ by means of $\lb(Y)=\bigsqcup_{y\in Y}A_y$, $Y\in\cP(X)$). 
		
		Assume moreover that $|K|>2$. Given a unital linear invertibility preserver $\vf:FI(X,K)\to FI(X,K)$, let $\lb$ be the corresponding endomorphism of $\cP(X)$. For any $y\in X$ there is a unique $x\in X$ such that $y\in\lb(\{x\})$. Then for any $\af\in FI(X,K)$ by \cref{vf(f)=sum-f(x_x)e_lb(x)+psi(f_J)} we have 
		\begin{align}\label{vf(f)(y_y)=f(x_x)}
			\vf(\af)_{yy}=\af_{xx},    
		\end{align}
		because $x\in L_{\af_{xx}}(\af)$ and $y\in\lb(\{x\})\sst \lb(L_{\af_{xx}}(\af))$. If we drop the assumption that $\vf$ is unital and apply the above argument to $\psi=\vf(\dl)\m\vf$, then we will get $\vf(\dl)_{yy}\af_{xx}$ on the right-hand side of \cref{vf(f)(y_y)=f(x_x)}. On the other hand, if we additionally assume that $\vf$ is strong, then $\lb$ is an automorphism of $\cP(X)$ by \cref{vf-strong<=>lb(A)-nonempty} (since $|X|<\infty$), so $\lb(\{x\})=\{\mu(x)\}$ for some bijection $\mu$ of $X$ by \cref{auto-P(X)}. It follows that
		\begin{align}\label{vf(f)(mu(y)_mu(y))=f(x_x)}
			\vf(\af)_{\mu(x)\mu(x)}=\af_{xx}    
		\end{align}
		for all $x\in X$. And in the non-unital case one would have $\vf(\dl)_{\mu(x)\mu(x)}\af_{xx}$ on the right-hand side of \cref{vf(f)(mu(y)_mu(y))=f(x_x)}. Thus, we get \cite[Theorem 3.3]{ChooiLim98} as a particular case.
	\end{rem}
	
	Now, we would like to see how our results agree with Theorem 1.1, Lemma 3.2 and Theorem 3.3 from \cite{Slowik15}.
	\begin{rem}\label{connection-with-Slowik}
		Let $X$ be countably infinite. If $|K|=\infty$, then $|K|$-complete endomorphisms of $\cP(X)$ are uniquely  determined by the corresponding partitions $\{\lb(\{x\})\}_{x\in X}$ of $X$ as in \cref{inv-pres-|X|<infty}. In this case, for any unital linear invertibility preserver $\vf:FI(X,K)\to FI(X,K)$, formula \cref{vf(f)(y_y)=f(x_x)} still holds, as well as its version for a non-unital $\vf$. However, \cref{vf(f)(mu(y)_mu(y))=f(x_x)} and its ``non-unital'' generalization only remain valid under the extra assumption that the strong preserver $\vf$ is surjective (see \cref{surj-strong-|K|>2,inj-non-surj-lb}). Thus, we get Theorem~1.1, Lemma~3.2 and Theorem~3.3 from \cite{Slowik15} (where the necessary condition $\bigsqcup_{n\in\NN}\mu(n)=\NN$ appears to be missing) as particular cases.
		
		If $|K|<\infty$, then any endomorphism $\lb$ of $\cP(X)$ is $|K|$-complete, but it is not in general determined by $\lb(\{x\})$, since it may not be $\aleph_0$-complete as we have seen in \cref{non-complete-endo-P(X)}. Hence, one cannot use formula \cref{vf(f)(y_y)=f(x_x)} to calculate $\vf(\af)_D$. Moreover, taking $\lb$ from \cref{non-complete-endo-P(X)}, one can easily construct a unital linear invertibility preserver $\vf$ for which \cref{vf(f)(y_y)=f(x_x)} does not make sense, as $\bigsqcup_{x\in X}\lb(\{x\})\ne X$ (for example, use \cref{vf(f)=sum-f(x_x)e_lb(x)+psi(f_J)} with $\psi(\af)=\af_J$). Observe that such a $\vf$ will be strong because its $\lb$ is injective (if $A\ne\emptyset$, then there is $a\in A$, so $\emptyset\ne\{\mu(a)\}=\lb(\{a\})\sst \lb(A)$). Thus, we see that the statements of Theorem 1.1 and Theorem 3.3 from \cite{Slowik15} do not hold in this case.
	\end{rem}
	\subsection{Invertibility preservers that preserve the inverses}\label{sec-inverse-pres}
	
	Let $\vf:A\to B$ be an invertibility preserver. We say that $\vf$ \textit{preserves the inverses} or is an \textit{inverse preserver}, if for any $a\in U(A)$ one has $\vf(a\m)=\vf(a)\m$.
	
	\begin{prop}\label{vf-pres-inverses=>vf(1)vf-pres-idemp}
		Let $A$ and $B$ be $K$-algebras, where $\ch(K)\ne 2$. If a linear invertibility preserver $\vf:A\to B$ preserves the inverses, then $\psi=\vf(1_A)\vf$ preserves the idempotents. 
	\end{prop}
	\begin{proof}
		First observe that $\vf(1_A)=\vf(1_A)\m$ and hence 
		\begin{align}\label{vf(1_A)^2=1_B}
			\vf(1_A)^2=1_B.
		\end{align}
		Take $e\in E(A)$. Then $1_A+e\in U(A)$ with $(1_A+e)\m=1_A-\frac 12e$. Therefore,
		\begin{align}
			(\vf(1_A)+\vf(e))(\vf(1_A)-\frac 12\vf(e))=(\vf(1_A)-\frac 12\vf(e))(\vf(1_A)+\vf(e))=1_B,
		\end{align} 
		which in view of \cref{vf(1_A)^2=1_B} gives
		\begin{align}\label{vf(e)vf(q)-half.vf(1)vf(a)=half.vf(e)^2}
			\vf(e)\vf(1_A)-\frac 12\vf(1_A)\vf(e)=\frac 12\vf(e)^2=\vf(1_A)\vf(e)-\frac 12\vf(e)\vf(1_A).
		\end{align}
		Solving the system \cref{vf(e)vf(q)-half.vf(1)vf(a)=half.vf(e)^2} in $\vf(e)\vf(1_A)$ and $\vf(1_A)\vf(e)$, we get
		\begin{align}\label{vf(e)vf(1)=vf(1)vf(e)=vf(e)^2}
			\vf(e)\vf(1_A)=\vf(1_A)\vf(e)=\vf(e)^2.
		\end{align}
		Thus,
		\begin{align*}
			\psi(e)^2=\vf(1_A)\vf(e)\vf(1_A)\vf(e)=\vf(1_A)^2\vf(e)^2=\vf(e)^2=\vf(1_A)\vf(e)=\psi(e)
		\end{align*}
		by \cref{vf(e)vf(1)=vf(1)vf(e)=vf(e)^2,vf(1_A)^2=1_B}, so $\psi$ preserves the idempotents.
	\end{proof}
	
	\begin{cor}\label{vf-unital-pres-inverses=>vf-pres-idemp}
		Under the conditions of \cref{vf-pres-inverses=>vf(1)vf-pres-idemp}, if $\vf$ is unital, then $\vf$ preserves the idempotents.
	\end{cor}
	
	\begin{cor}\label{vf-pres-inverses=>vf-Jordan-homo}
		Let $|X|<\infty$ and $\ch(K)\ne 2$. Then $\vf:FI(X,K)\to FI(X,K)$ is a linear unital inverse preserver if and only if $\vf$ is a unital Jordan endomorphism. Thus, in this case $\vf$ is a near sum~\cite{Benkovic05} of an endomorphism and an anti-endomorphism of $FI(X,K)$ by \cite[Theorem 2.1]{Akkurts-Barker}.
	\end{cor}
	\begin{proof}
		The ``only if'' part follows from \cref{vf-unital-pres-inverses=>vf-pres-idemp} and \cite[Proposition 2.1]{GK}. The ``if'' part is by \cite[Proposition 1.3]{Sourour96}.
	\end{proof}
	
	The following corollary is a finite-dimensional analogue of \cite[Theorem 1.2]{Slowik15}.
	\begin{cor}\label{vf-pres-inverses=>vf-pm-auto-or-anti-auto}
		Let $X$ be a finite connected poset and $\ch(K)\ne 2$. Then $\vf$ is a bijective linear inverse preserver if and only if $\vf$ is either a $\pm$automorphism of $FI(X,K)$ or a $\pm$anti-automorphism of $FI(X,K)$.
	\end{cor}
	\begin{proof}
		Let $\vf$ be a bijective linear inverse preserver. By \cref{vf-pres-inverses=>vf(1)vf-pres-idemp} the map $\psi=\vf(\dl)\vf$ is a Jordan automorphism of $FI(X,K)$, so it is either an automorphism of $FI(X,K)$ or an anti-automorphism of $FI(X,K)$ by \cite[Theorem 4.4]{GK}. Since $FI(X,K)$ is spanned by idempotents (say, by $\{e_x\}_{x\in X}\sqcup\{e_x+e_{xy}\}_{x<y}$) and $\vf$ is bijective, it follows from \cref{vf(e)vf(1)=vf(1)vf(e)=vf(e)^2} that $\vf(\dl)$ is central, whence $\vf(\dl)=k\dl$ for some $k\in K$ by \cite[Theorem 1.3.13]{SpDo}. By \cref{vf(1_A)^2=1_B} we have $k=\pm 1$, so that $\psi=\pm\vf$.
		
		The converse is obvious.
	\end{proof}
	
	The following example shows that \cref{vf-pres-inverses=>vf-Jordan-homo} does not hold over $\Z_2$.
	\begin{exm}
		Let $X=\{1,2,3\}$ with the usual order. Consider the linear map $\vf:FI(X,\Z_2)\to FI(X,\Z_2)$ given by $\vf(e_1)=e_1$, $\vf(e_2)=e_1+e_2$, $\vf(e_3)=e_1+e_3$, $\vf(e_{12})=e_{12}$, $\vf(e_{13})=\vf(e_{23})=0$. Since $(\dl+e_{12})^2=(\dl+e_{13})^2=(\dl+e_{23})^2=(\dl+e_{12}+e_{13})^2=(\dl+e_{13}+e_{23})^2=(\dl+e_{12}+e_{23})(\dl+e_{12}+e_{13}+e_{23})=\dl$, it is easy to see that $\vf$ is a unital inverse preserver. However, $\vf(e_1\circ e_{12})=\vf(e_{12})=e_{12}\ne 0=(e_1+e_2)\circ e_{12}=\vf(e_1)\circ\vf(e_{12})$, so $\vf$ is not a Jordan endomorphism of $FI(X,\Z_2)$.
	\end{exm}
	
		
	\subsection{The case $K=\Z_2$}\label{sec-inv-pres-over-Z_2}
	Observe that linearity of maps over $\Z_2$ is the same as their additivity. Then it turns out that, whenever $K=\Z_2$, the map $\lb:\cP(X)\to \cP(X)$ from \cref{from-vf-to-lb} preserves a weaker structure on $\cP(X)$ than that of a Boolean algebra. Namely, recall that $\cP(X)$ is an elementary abelian $2$-group under the symmetric difference $A\tr B=(A\setminus B)\sqcup(B\setminus A)$. It is in fact isomorphic to the direct power $(\Z_2)^X$, i.e. the group of maps $X\to\Z_2$ under the coordinate-wise addition.
	
	\begin{lem}\label{lb-prese-symm-diff}
		Let $\vf:FI(X,\Z_2)\to FI(X,\Z_2)$ be a unital linear in\-ver\-ti\-bi\-li\-ty preserver and $\lb:\cP(X)\to \cP(X)$ the associated map as in \cref{from-vf-to-lb}. Then $\lb$ is an endomorphism of the abelian group $(\cP(X),\tr)$ with $\lb(X)=X$. 
	\end{lem}
	\begin{proof}
		Since $e_A+e_B=e_{A\tr B}$ for all $A,B\sst X$ and $\vf$ is additive, then $\lb(A\tr B)=\lb(A)\tr\lb(B)$. Moreover, $\dl=e_X$ implies $\lb(X)=X$.
	\end{proof}
	
	\begin{thrm}\label{inv-pres-over-Z_2}
		The unital linear invertibility preservers of $FI(X,\Z_2)$ are exactly the maps of the form 
		\begin{align}\label{vf(f)=e_lb(L_1(x))+psi(f)}
			\vf(\af)=e_{\lb(L_1(\af))}+\psi(\af),
		\end{align}
		where $\lb$ is an endomorphism of $(\cP(X),\tr)$ with $\lb(X)=X$ and $\psi:FI(X,\Z_2)\to J(FI(X,\Z_2))$ is a linear map annihilating $\dl$.
	\end{thrm}
	\begin{proof}
		Let $\vf$ be given by \cref{vf(f)=e_lb(L_1(x))+psi(f)}. For all $\af,\bt\in FI(X,\Z_2)$ we have
		\begin{align}\label{vf(f)_D+vf(g)_D=e_(lb(L_1(f))-tr-lb(L_1(g)))}
			\vf(\af)_D+\vf(\bt)_D=e_{\lb(L_1(\af))}+e_{\lb(L_1(\bt))}=e_{\lb(L_1(\af))\tr \lb(L_1(\bt))}=e_{\lb(L_1(\af)\tr L_1(\bt))}.
		\end{align}
		Since 
		\begin{align*}
			e_{L_1(\af)\tr L_1(\bt)}=e_{L_1(\af)}+e_{L_1(\bt)}=\af_D+\bt_D=(\af+\bt)_D=e_{L_1(\af+\bt)},
		\end{align*}
		we conclude from \cref{vf(f)=e_lb(L_1(x))+psi(f),vf(f)_D+vf(g)_D=e_(lb(L_1(f))-tr-lb(L_1(g)))} that $\vf(\af+\bt)_D=\vf(\af)_D+\vf(\bt)_D$. Then
		\begin{align*}
			\vf(\af+\bt)=\vf(\af+\bt)_D+\psi(\af+\bt)=\vf(\af)_D+\vf(\bt)_D+\psi(\af)+\psi(\bt)=\vf(\af)+\vf(\bt),
		\end{align*}
		so $\vf$ is linear. If $\af\in U(FI(X,\Z_2))$, then $L_1(\af)=X$, so $\lb(L_1(\af))=X$, whence $\vf(\af)\in U(FI(X,\Z_2))$. Thus, $\vf$ is an invertibility preserver. It is unital, because $\psi(\dl)=0$.
		
		Conversely, let $\vf$ be a unital linear invertibility preserver of $FI(X,\Z_2)$. Then $\vf(\af)_D=\vf(\af_D)_D=\vf(e_{L_1(\af)})_D=e_{\lb(L_1(\af))}$ by \cref{vf(f)_D-is-vf(f_D)_D,from-vf-to-lb}, where $\lb$ is an endomorphism of $(\cP(X),\tr)$ with $\lb(X)=X$ by \cref{lb-prese-symm-diff}. As in the proof of \cref{inv-pres-for-|K|>2} we define $\psi(\af)=\vf(\af)_J$ and complete thus the proof of \cref{vf(f)=e_lb(L_1(x))+psi(f)}.
	\end{proof}
	
	We have a description of strong invertibility preservers of $FI(X,\Z_2)$ analogous to that of \cref{vf-strong<=>lb(A)-nonempty}.
	\begin{cor}\label{vf-strong<=>lb-injective}
		Under the conditions of \cref{inv-pres-over-Z_2}, the unital invertibility preserver $\vf$ is strong if and only if the endomorphism $\lb$ is injective. 
	\end{cor}
	\begin{proof}
		The ``only if'' part is proved the same way as the ``only if'' part of \cref{vf-strong<=>lb(A)-nonempty} with the unique difference that we do not need \cite[Lemma 5.3]{Monk} anymore to justify \cref{lb(A)-empty=>A-empty}. For the ``if'' part assume the injectivity of $\lb$. If $\vf(\af)\in U(FI(X,\Z_2))$, then $e_{\lb(L_1(\af))}=\vf(\af)_D=\dl$, so $\lb(L_1(\af))=X$. Since $\lb(X)=X$ and $\lb$ is injective, then $L_1(\af)=X$, i.e. $\af\in U(FI(X,\Z_2))$.
	\end{proof}
	
	The following corollary is proved similarly to \cref{bij-strong-|K|>2}.
	\begin{cor}\label{bij-strong-over-Z_2}
		Under the conditions of \cref{inv-pres-over-Z_2}, the unital invertibility preserver $\vf$ is surjective (resp. bijective) and strong if and only if
		\begin{enumerate}
			\item\label{lb-automorphism} the associated map $\lb$ is an automorphism of the abelian group $(\cP(X),\tr)$,
			\item\label{psi-bijection-of-J} $\psi$ maps (resp. maps bijectively) $J(FI(X,\Z_2))$ onto $J(FI(X,\Z_2))$.
		\end{enumerate}
	\end{cor}

			\section*{Acknowledgements}
			Jorge J. Garcés was partially supported by grant PID2021-122126NB-C31 funded by MCIN/AEI/10.13039/501100011033 and by ERDF/EU and Junta de Andalucía grant FQM375. Mykola Khrypchenko was partially supported by CMUP, member of LASI, which is financed by national funds through FCT --- Fundação para a Ciência e a Tecnologia, I.P., under the project with reference UIDB/00144/2020. The authors are grateful to Professor Matej Bre\v{s}ar for providing them with a pdf of the paper~\cite{BresarSemrl99}.
			
			%
			
			\bibliography{bibl}{}
			\bibliographystyle{acm}
			
		\end{document}